\documentclass[12pt]{amsart}
\usepackage{ifpdf}
\ifpdf\usepackage[colorlinks]{hyperref} 
\usepackage{microtype}
\else\usepackage[hypertex]{hyperref}\fi 

\newcommand{\C}{{\mathbb C}}
\newcommand{\Z}{{\mathbb Z}}
\newcommand{\R}{{\mathbb R}}

\newtheorem{theorem}{Theorem}[section]
\newtheorem{lemma}[theorem]{Lemma}
\newtheorem{proposition}[theorem]{Proposition}

\theoremstyle{definition}

\begin{document}
\title{Bi-Lipschitz geometry of complex surface singularities}
\author{Lev Birbrair}
\address{Departamento de Matem\'atica, Universidade Federal do Cear\'a
(UFC), Campus do Picici, Bloco 914, Cep. 60455-760. Fortaleza-Ce,
Brasil} \email{birb@ufc.br}
\author{Alexandre Fernandes}
\address{Departamento de Matem\'atica, Universidade Federal do Cear\'a
(UFC), Campus do Picici, Bloco 914, Cep. 60455-760. Fortaleza-Ce,
Brasil} \email{alex@mat.ufc.br}
\author{Walter D.
  Neumann} \address{Department of Mathematics, Barnard College,
  Columbia University, New York, NY 10027}
\email{neumann@math.columbia.edu}


\begin{abstract}
  We discuss the bi-Lipschitz geometry of an isolated singular point
  of a complex surface which particular emphasis on when it is 
  metrically conical.
\end{abstract}

\maketitle
\section{Introduction}
A very basic question in metric geometry is whether a neighborhood
of a point in an algebraic or semialgebraic set $V$ is metrically
conical, i.e., bi-Lipschitz equivalent to a metric cone (our metric is
always the ``inner metric,'' given distance within $V$, rather than
``outer metric,'' given by distance in the ambient affine space).

For real algebraic sets an extensive literature exists on local
bi-Lipschitz geometry, and failure of metric conicalness is
common. The characteristic example is the \emph{$\beta$--horn} for
$\beta=\frac pq\ge1$,
$$\{(x,y,z)\in \R^3: (x^2+y^2)^q=z^{2p},z\ge0\}\,,$$ 
which is topologically the cone on a circle but is bi-Lipschitz
classified by $\beta$ \cite{B99} and is thus not metrically
conical if $\beta>1$. The full  bi-Lipschitz classification of
germs semi-algebraic sets in this dimension was completed in \cite{B99} and
\cite{BF00}. 

However, the techniques that have been used to distinguish local
bi-Lipschitz geometry in the real algebraic case are
mostly useless in the complex case.  For example, the $\beta$--horns
can be distinguished by their \emph{volume growth number}, defined as
$$\mu(V,p):=\sup \,\Bigl\{r>0:\lim_{\epsilon\to
  0}\,\frac{\operatorname{Vol}(V\cap
  B_\epsilon(p))}{\epsilon^r}=0\Bigr\}.$$ This is a well defined
rational number for any semialgebraic germ $(V,p)$ (Lion-Rolin
\cite{LR}), and is bi-Lipschitz invariant (Birbrair and Brasselet
\cite{BB00}), and equals $\beta+1$ for the $\beta$--horn.  On the other
hand, a complex variety of dimension $n$ has volume growth number
equal to $2n$ at every point. And, in fact, complex algebraic curves are
metrically conical at every point.

It is worth stressing that the local geometry is a bi-Lipschitz
invariant of a complex analytic germ $(V,p)$ (independent of
embedding). For if one uses any set of generators of the local ring
$\mathcal O_p(V)$ to embed $(V,p)$ in some $(\C^N,0)$, then $(V,p)$
inherits a Riemannian metric at smooth points of $V$ which gives a
distance metric on $(V,p)$ that is unchanged up to bi-Lipschitz
equivalence when adding to the set of generators of $\mathcal O_p(V)$
that is used.

\medskip 
The first examples of failure of metric conicalness in the complex
setting were given by the first two authors in \cite{BF08}. They
demonstrated that the $A_k$--singularity $z^{k+1}=x^2+y^2$ is not
metrically conical for $k$ odd and $\ge3$.
$A_k$ is weighted homogeneous with weights
$(\frac{k+1}2,\frac{k+1}2,1)$.  The 
current authors showed much more generally:
\begin{theorem}[\cite{BFN08}]\label{th:1}
  A weighted homogeneous surface singularity is not metrically conical
   if its two lowest weights are distinct.
\end{theorem}

A singularity is \emph{homogeneous} if it is weighted homogeneous with
all its weights equal. Homogeneous singularities are (obviously)
metrically conical; the converse holds for cyclic quotient
singularities:
\begin{theorem}[\cite{BFN08}]\label{th:2}
  A cyclic quotient singularity $\C^2/(\Z/n)$ is metrically
  conical if and only if it is homogeneous.
\end{theorem}

Since the two lowest weights are equal to each other for many
non-homogeneous cyclic quotient singularities, the converse of Theorem
\ref{th:1} is not necessarily true.  But we will prove here a converse
to Theorem \ref{th:1} for Brieskorn hypersurfaces:
\begin{theorem}\label{th:3}
The Brieskorn singularity $$V(a,b,b):=\{(x,y,z)\in \C^3:
\alpha_0x^a+\alpha_1y^b+\alpha_2z^b=0\}$$ with $a<b$ has a metrically
conical singularity at $0$ for any $\alpha_0,\alpha_1,\alpha_2\in\C-\{0\}$.   
\end{theorem}

We close this introduction by sketching the known obstructions to
metric conicalness in the complex case. 

Let $M_\epsilon$ be the link of the point $p\in V$ (the boundary of
$V\cap B_\epsilon(p)$, where $B_\epsilon(p)$ is an $\epsilon$-ball in
some ambient affine space in which $(V,p)$ is embedded, $\epsilon$
sufficiently small).  A non-trivial homology class or free homotopy
class in $M_\epsilon$ will have a lower bound on the diameter of any
cycle representing it in $M_\epsilon$. In a metric cone this lower
bound will shrink at most linearly with respect to $\epsilon$ as
$\epsilon \to 0$. If such a cycle exists which shrinks faster than
linearly, it therefore obstructs metric conicalness. Theorems
\ref{th:1} and \ref{th:2} are proved by exhibiting such \emph{fast
  cycles} for $\pi_1(M)$. This is close to the ideas of metric
homology \cite{BB00,BB02}, but the link of a surface singularity may
be a homology sphere, in which case homology is not useful. Instead,
one might see this as a first step to ``metric homotopy theory'' in
bi-Lipschitz geometry.

A main tool used in \cite{BF08} was to exhibit a ``Cheeger cycle'' in
$V$ (also called ``separating set''), a codimension $1$ subset that
divides $V$ into pieces of roughly equal volume, but whose
$3$-dimensional volume shrinks faster towards $p$ than it could in a
metric cone. In \cite{BF08} the Cheeger cycle was constructed as the
union of orbits of the real points of $V$ under the $\C^*$--action, so
it was important that an appropriate real form be used. At the time it
was also only known that this Cheeger cycle obstructs $(V,p)$ being
\emph{semi-algebraicly} bi-Lipschitz equivalent to a metric cone.  We are
grateful to Bruce Kleiner for showing us how the semi-algebraic
condition on the bi-Lipschitz equivalence can be removed.

In the final section of this paper we revisit the separating set
approach, describing a more robust version using ``conflict sets.'' We
then sketch how it can be used to show that the Brian\c{c}on-Speder
family of singularities of constant topological type does not have
constant bi-Lipschitz type. The details of this argument will appear
in a future paper \cite{BFN-inprep}.  The embedded real case of this,
which is much more elementary, was proved in \cite{koike03}.

\smallskip\noindent\textbf{Acknowledgements.} The authors acknowledge
support for this research under the following grants: CNPq grant no
300985/93-2 (Birbrair), CNPq grant no 300393/2005-9 (Fernandes), NSA
grant H98230-06-1-011 and NSF grant no.\ DMS-0206464
(Neumann). Birbrair and Neumann express their gratitude to the ICTP in
Trieste for its hospitality during the final work on this paper.

\section{Conical Brieskorn singularities}
This section is devoted to the proof of Theorem \ref{th:3}; we refer
to the notation of that theorem.  Since changing the coefficients
$\alpha_i$ can be realized by a linear change of coordinates in
$\C^3$, which is  bi-Lipschitz, we can choose the
coefficients at our convenience. We choose
$$V=\{(x,y,z)\in \C^3:
x^a+y^b-z^b=0\}\,.$$ Then the projection of $V$ to the $(y,z)$-plane is
an $a$--fold branched cover branched along the lines $z=\omega y$
with $\omega\in\mu_b$, the $b$--th root of unity.  We will show that
this projection has a bounded Lipschitz constant except in a thin
neighborhood of the branch locus. The conical structure of the
$(y,z)$--plane pulls back to $V$ except in these thin
neighborhoods. We then show that these neighborhoods can be chosen
to also carry a
conical structure.

We first decompose $V$ into pieces. When we
refer to $V$ and its pieces we will really mean the germ at $0$, so
$V$ is always implicitly intersected with a small neighborhood of
$0\in \C^3$.

Our two pieces will consist of a disk-bundle neighborhood of the
branch set of the projection to the $(y,z)$--plane and the closure of
its complement, and we will show that both pieces are metrically
conical at $0$. But we start with a preliminary decomposition into two
pieces which are not metrically conical.

Our preliminary decomposition of $V$ is as follows:
\begin{align*}
  V_0&:=\{(x,y,z)\in V: |x|^{2a-2}\ge |y|^{2b-2}+|z|^{2b-2}\}\\
  V_1&:=\{(x,y,z)\in V: |x|^{2a-2}\le |y|^{2b-2}+|z|^{2b-2}\}\,.
\end{align*}
Using the fact that $x^a=-(y^b-z^b)$ on $V$, we can write this as 
\begin{align*}
  V_0&:=\{(x,y,z)\in V: |y^b-z^b|^{(2-2/a)}\ge |y|^{2b-2}+|z|^{2b-2}\}\\
  V_1&:=\{(x,y,z)\in V: |y^b-z^b|^{(2-2/a)}\le |y|^{2b-2}+|z|^{2b-2}\}\,.
\end{align*}
Denote the images of the projections to the $(y,z)$--plane by
\begin{align*}
 W_0&:=\{(y,z)\in \C^2: |y^b-z^b|^{(2-2/a)}\ge
|y|^{2b-2}+|z|^{2b-2}\}\\
 W_1&:=\{(y,z)\in \C^2: |y^b-z^b|^{(2-2/a)}\le
|y|^{2b-2}+|z|^{2b-2}\}\,.
\end{align*}
\begin{lemma}\label{le:1}
  The projection of $V$ to the $(y,z)$--plane is an $a$--fold cyclic
  covering branched along the lines $z=\omega y$ with $\omega\in
  \mu_a$ (the $a$--th roots of unity). When restricted to $V_0-\{0\}$
  it is a bi-Lipschitz unramified covering of its image
  $W_0-\{0\}$ with Lipschitz constant $\sqrt{1+\frac{ b^2}{a^2}}$.
\end{lemma}
\begin{proof}
  Write $f(x,y,z)=x^a+y^b-z^b$ so $V=f^{-1}(0)$.  That $V\to \C^2$ is
  a cyclic branched cover follows because it is the orbit map of
  $\mu_a$ acting on $V$ by multiplication in the $x$-coordinate.
  Branching is thus along $x=0$ which projects to the set $y^b=z^b$ in
  $\C^2$. This is the set $\{(y,z):z=\omega y, \omega\in \mu_a\}$. The
  restriction to $V_0$ is an unramified cover since $W_0-\{0\}$ does
  not intersect this branch locus.

  The bi-Lipschitz constant of the projection at a point of $V$ will
  be the bi-Lipschitz constant of the projection of the tangent plane
  at that point to the $(y,z)$--plane.  The tangent plane at the point
  is given by the orthogonal complement of the complex gradient
  $\overline{\nabla f}$. The following lemma is an exercise:
  \begin{lemma}
    For planes
  complex orthogonal to unit vectors $u_1$ and $u_2$ the orthogonal
  projection of one plane to the other has bi-Lipschitz constant
  $1/|\langle u_1,u_2\rangle|$, where $\langle,\rangle$ is hermitian
  inner product.  \qed
  \end{lemma}
 Returning to the proof of Lemma \ref{le:1}, the unit vectors in question
  are
  \begin{equation}
    \label{eq:grad}
   \frac{\overline{\nabla f}}{|{\nabla f}|}= \frac{(a\bar
    x^{a-1},b\bar y^{b-1},-b\bar
    z^{b-1})}{\sqrt{a^2|x^{a-1}|^2+b^2(|y^{b-1}|^2+|z^{b-1}|^2)}}
  \end{equation}
  and $(1,0,0)$. So the bi-Lipschitz constant is
$$\frac{\sqrt{a^2|x^{a-1}|^2+b^2(|y^{b-1}|^2+|z^{b-1}|^2)}}{|ax^{a-1}|}
\le
\frac{|\sqrt{a^2|x^{a-1}|^2+ b^2|x^{a-1}|^2}}{|ax^{a-1}|}
=\sqrt{1+\frac{b^2}{a^2}}\,,$$
where the inequality uses the defining inequality for $V_0$.  
\end{proof}

Note that $W_0-\{0\}$ and $W_1-\{0\}$ decompose $\C^2-\{0\}$ into two
subsets that meet along their boundaries. We claim:
\begin{lemma}\label{le:2} Assume $b> a$. Then, in a neighborhood of $0$, 
  $W_1-\{0\}$ consists of disjoint closed disk-bundle neighbourhoods
  of the lines $z=\omega y$. At small distance $r$ from the origin
  these disks have radius close to
  $cr^{(b-1)/(a-1)}$ for some $c>0$.
\end{lemma}
\begin{proof}
  Since the $\mu_b$-action that multiplies just the $z$--coordinate by
  $\omega$ permutes the lines in question, it suffices to consider the
  line $y=z$. For fixed $y=v$ a transverse section to this line can be
  given by $(v(1-\xi),v(1+\xi))$ as $\xi$ varies. We restrict $\xi$ to
  be small so we are in a neighbourhood of the line and we consider
  the set of $(v(1-\xi),v(1+\xi))\in \C^2$ satisfying 
$$|(v^b(1-\xi)^b-v^b(1+\xi)^b)|^{(2-2/a)}\le
   |v(1-\xi)|^{2b-2}+|v(1+\xi)|^{2b-2}\,.$$
This inequality simplifies to:
$$|(1-\xi)^b-(1+\xi)^b|^{2(a-1)/a}\le
|v|^{2(b-a)/a}(|1-\xi|^{2b-2}+|1+\xi|^{2b-2})$$
To first order in $\xi$ this is
$$|2b\xi|^{2(a-1)/a}\le2|v|^{2(b-a)/a}\,,$$ which gives
$$\sqrt2|v\xi|\le c|v|^{(b-1)/(a-1)}\,,$$ for some constant
$c$ (specifically, $c=\sqrt2/(2^{(a-2)/(2a-2)}b)$). 

Since the radius of the transverse section at $v$ is the
maximum of $\sqrt2 |v\xi|$, the lemma follows.
 \end{proof}

 Denote the branch locus $y^b=z^b$ in $V$ or in the $(y,z)$--plane by
 $B$ (we use the same notation for both). 

 Note that the radius of the disk-bundle neighbourhood $W_1$ of $B$ at
 distance $r$ from the origin is of order $r^{(b-1)/(a-1)}$, which is
 $o(r)$.
 Thus if we choose a small $\delta>0$ and decompose $V$ conically with
 respect to the $(y,z)$--plane as follows:
\begin{align*}
  C_0&:=\{(x,y,z)\in V: d((y,z),B)\ge \delta|(y,z)|\}\\
  C_1&:= \{(x,y,z)\in V: d((y,z),B)\le \delta|(y,z)|\}\,,
\end{align*}
with images in $\C^2$:
\begin{align*}
 D_0&:=\{(y,z)\in \C^2:d((y,z),B)\ge\delta|(y,z)|\}\\
 D_1&:=\{(y,z)\in \C^2:d((y,z),B)\le\delta|(y,z)|\}\,,
\end{align*}
then, so long as we are in a small enough neighborhood of $0\in V$,
the sets $C_0$ and $D_0$ are subsets of $V_0$ and $W_0$.  Since $D_0$
is strictly conical and $C_0$ is a bi-Lipschitz covering of it, $C_0$
is metrically conical.  To complete the proof of the theorem we must
just show that the other piece, $C_1$, is also metrically conical,
since it follows from Corollary 0.2 of \cite{valette} that the union
is then metrically conical.

$C_1$ is a union of disk-bundle neighbourhoods of the lines $y=\omega
z$. As before, it suffices to focus just on the component $C'_1$,
which is a neighborhood of $y=z$. The proof of Theorem \ref{th:3} is
then completed by the following lemma.
\begin{lemma}\label{le:4}
  The map of $C'_1$ to $\C^2$ given by $$(x,y,z)\mapsto
  \Bigl(\frac{y+z}{2}~,~ e^{i\arg
    x}\sqrt{|x|^2+\frac{|y-z|^2}4}\Bigr)$$ is a bi-Lipschitz
  homeomorphism onto a metric cone in $\C^2$.
\end{lemma}
The rest of this section is devoted to proving this lemma. We first
introduce more convenient coordinates on $C'_1$. Define
$$
  u:=\frac{z-y}{2}\qquad v:=\frac{z+y}{2}
$$
The inverse change of coordinates is 
$$
  y:={v-u}\qquad z:={v+u}
$$
Thus $v$ is as in the proof of Lemma \ref{le:2} and $u$ equals $v\xi$
in the notation of that proof. 
$C'_1$ is given by
\begin{equation}
  \label{eq:c}
  C'_1=\{(x,u,v)\in \C^3:x^a=(v+u)^b-(v-u)^b,|u|\le\delta|v|\}
\end{equation}
and the map of Lemma \ref{le:4} is 
\begin{equation}
  \label{eq:map}
F(x,u,v) =
  \Bigl(r_ve^{i\theta_v}~,~ e^{i\theta_x}\sqrt{r_x^2+r_u^2}\,\Bigr)\,,
\end{equation}
where we are using polar coordinates
$$x=r_xe^{i\theta_x},\quad u=r_ue^{i\theta_u},\quad
v=r_ve^{i\theta_v}\,.$$
Since for $v$ small the size of $x$ is negligible with respect to
$\delta|v|$, the image of $F$ is extremely close to the conical set
$\{(u,v):|u|\le\delta v\}$, and the main issue is checking the
bi-Lipschitz bound.

We can rewrite the defining equation $x^a=(v+u)^b-(v-u)^b$ in
\eqref{eq:c} as
\begin{equation}
  \label{eq:defeq}
  x^a=2buv^{b-1} +u^3g(u,v),
\end{equation}
with $g(u,v)$ a polynomial of degree $b-3$.  If we choose $\delta$
very small in \eqref{eq:c} then $u^3g(u,v)$ will be virtually
negligible, so to simplify calculation we will omit this term for now
and work with
\begin{equation}
  \label{eq:C}
  C:=\{(x,u,v)\in \C^3:x^a=2buv^{b-1},|u|\le\delta|v|\}
\end{equation} 
instead of $C'_1$. It is fairly clear that up to bi-Lipschitz equivalence this
changes nothing, but we come back to this issue later.

It is helpful to first consider a single
transverse section to the $v$-line, so, writing $D=1/(2bv^{b-1})$, we
start by proving
\begin{lemma}\label{le:5}
  For any integer $a\ge 1$ and any $D\in \C^*$ the map $f$
of the graph
$  S_D:=\{(x,u):u=Dx^a\}
$
to $\C$ given by 
\begin{equation}
  \label{eq:f}
  f(x,u)=e^{i\theta_x}\sqrt{r_x^2+r_u^2}=e^{i\theta_x}
\sqrt{r_x^2+|D|^2r_x^{2a}}\,,
\end{equation}
is a bi-Lipschitz homeomorphism with bi-Lipschitz bound $\le a$.
\end{lemma}
\begin{proof}
The metric in $S_D$ is given by
\begin{equation}
  \label{eq:m1}
\begin{aligned}
    ds^2&=|dx|^2+|du|^2=|dx|^2 + |aDx^{a-1}|^2|dx|^2\\
&=({dr_x}^2+r_x^2{d\theta_x}^2)(1+a^2|D|^2r_x^{2a-2})\,.
\end{aligned}
\end{equation}
On the other hand, differentiating $f$ gives
\begin{equation}
  \label{eq:df}
  df=e^{i\theta_x}\Bigl(\frac{r_x+a|D|^2r_x^{2a-1}}
{\sqrt{r_x^2+|D|^2r_x^{2a}}}\,dr_x+
i\sqrt{r_x^2+|D|^2r_x^{2a}}\,d\theta_x\Bigr)\,,
\end{equation}
so the metric pulled back by $f$ is 
\begin{equation}
  \label{eq:m2}
  |df|^2=\frac{(1+a|D|^2r_x^{2a-2})^2}{1+|D|^2r_x^{2a-2}}\;{dr_x}^2 ~+~ (1+|D|^2r_x^{2a-2})\;r_x^2{d\theta_x}^2\,.
\end{equation}
The ratio of coefficients of $r_x^2{d\theta_x}^2$ in \eqref{eq:m1} and
\eqref{eq:m2} increases steadily from $1$ to $a^2$ as $r_x$ increases
from $0$ to $\infty$.  And it is an exercise to check that the ratio
of the coefficient of ${dr_x}^2$ increases from $1$ to
$\frac{(a+1)^2}{4a}$ and then decreases again to $1$ as
$|D|^2r_x^{2a-2}$ goes from $0$ to $\infty$ via $1/a$.  Since
$\frac1{a^2}\le\frac{(a+1)^2}{4a}\le a^2$ for all $a\ge1$, it follows
that the ratio $ds^2/|df|^2$ is bounded below and above by
$\frac1{a^2}$ and $a^2$, so the bi-Lipschitz constant of $f$ is
bounded by $a$.
\end{proof}
Consider now a point $(x_0,u_0,v_0)\in C$. At this point we have two
surfaces: the surface $C=\{x^a=2buv^{b-1}\}$ that we are interested
in, and the surface
$$C_{(x_0,u_0,v_0)}:=\{(x,u,v): x^a=2buv_0^{b-1}\}$$
which is the product of the $v$-plane with the curve of Lemma
\ref{le:5}.  For each of these surfaces we can consider the local
bi-Lipschitz constant of the map
$F(x,u,v)=(v,e^{i\theta_x}\sqrt{r_x^2+r_u^2})$. For $C_{(x_0,u_0,v_0)}$
we have already computed this and it is uniformly bounded by $a$. The
constants for
$C$ and $C_{(x_0,u_0,v_0)}$ have ratio bounded by the bi-Lipschitz
  constant of the projection between the tangent spaces of these two
  surfaces at the given point. So it remains to compute this latter
  number and show it is uniformly bounded.

  The gradients of the two surfaces at the given point are
  $$(ax_0^{a-1}, 2bv_0^{b-1}, 2b(b-1)u_0v_0^{b-2})\quad\text{and}\quad
  (ax_0^{a-1},
  2bv_0^{b-1},0)$$ respectively.  Referring to Lemma \ref{le:2} we see
  that the number in question is
  \begin{equation}
    \label{eq:ick}
        \frac{\sqrt{|ax_0^{a-1}|^2+ |2bv_0^{b-1}|^2+
        |2b(b-1)u_0v_0^{b-2}|^2}} {\sqrt{|ax_0^{a-1}|^2+
        |2bv_0^{b-1}|^2}}
  \end{equation}
  Note that the additional term in the numerator is at most
  $((b-1)\delta)^2$ times the term preceding it, so the whole
  expression in \eqref{eq:ick} is bounded by
  $\sqrt{1+((b-1)\delta)^2}$. This completes the proof of this point.

  Finally, we promised to revisit the issue of replacing $C'_1$ by $C$
  at equation \eqref{eq:C}. This is a similar argument to the one we
  have just completed --- one checks that the projection in the
  $x$--direction between $C$ and $C'_1$ has bi-Lipschitz constant bounded 
by $(1+\delta^2g(\delta))$ for a certain fixed degree $b-3$ polynomial
in $\delta$, and can hence be made arbitrarily close to 1 by choosing
$\delta$ small enough.
\section{Separating  sets and the Brian\c{c}on-Speder family}

For an algebraic germ $(X,p)$, the
\emph{$r$--density} of $X$ at $p$ is defined as
$$\operatorname{density}_r(X,p)=\lim_{\epsilon\to 0}\frac
{\operatorname{vol}(X\cap B_\epsilon(p))}{\epsilon^r}$$
Thus, the volume growth number, defined in the Introduction, is
$$\mu(X,p)=\sup\{r>0:\operatorname{density}_r(X,p)=0\}\,.$$ 
Although the value of the $r$--density is not generally a bi-Lipschitz
invariant, its vanishing or non-vanishing is, which is why $\mu$ also
is a bi-Lipschitz invariant. If $r=\dim X$ we speak simply of the
\emph{density}. 

Let $(X,p)$ be an $n$-dimensional germ and
$(Y,p)\subset (X,p)$. Then $(Y,p)$ is called a \emph{separating set} if
\begin{itemize}
\item $Y$ divides $X$ into two pieces $X_1$ and $X_2$, each containing $p$;
\item $\operatorname{density}_{n-1}(Y)=0$, $\operatorname{density}_{n}(X_1)\ne0$, $\operatorname{density}_{n}(X_2)\ne0$. 
\end{itemize}
In view of the comments above, the existence of a separating set is a
bi-Lipschitz invariant. It is also an obstruction to metric
conicalness:
\begin{proposition}
  A metric cone cannot contain a separating set.
\end{proposition}
With extra conditions on semialgebraicity of the conical structure and
sets in question this is in \cite{BF08}. Bruce Kleiner showed us how
to eliminate the extra assumptions; details will appear in
\cite{BFN-inprep}.

In \cite{BF08} a separating set was used to show the non-conicalness
of the $A_k$ singularity for odd $k\ge 3$. We here give a different
construction that works for all $k\ge 2$. Note that these are also
covered by Theorem \ref{th:1}, where the proof is by fast cycles; we
describe the separating set approach to illustrate its
usefulness.

Consider therefore $A_k$, written in the form
$V=\{(x,y,z):z^{k+1}=xy\}$. The subset $\{z=0\}\cap V$ has components
$L_1=\{x=0\}\cap V$ and $L_2=\{y=0\}\cap V$. Consider their
\emph{conflict set}
$$Y=\{p\in V:d(p,L_1)=d(p,L_2)\,,$$
where $d()$ is distance. A symmetry argument shows that $Y$ separates
$V$ into two isometric pieces $V_1$ and $V_2$, so 
$$\operatorname{density}_4(V_1)=\operatorname{density}_4(V_2)=\frac12
\operatorname{density}_4(V)\ne 0\,.$$ Any smooth arc on $Y$ that
approaches $0$ has $|x|=|y|=|z|^{(k+1)/2}$ and must therefore approach
tangent to the $z$--axis, so the tangent cone of $Y$ is the
$z$--axis. However, if a semialgebraic set $Y$ of dimension $n$ has
tangent cone of lower dimension, then $\operatorname{density}_n(Y)=0$
(see Federer \cite{federer}). Thus $Y$ is a separating set.

The Brian\c{c}on-Speder family \cite{BS75}
$$V_t=\{(x,y,z): z^{15}+zy^7+x^5+txy^6=0\}$$ is a family of weighted
homogeneous surface singularities with weights $(3,2,1)$ which is
topologically a trivial family but which does not satisfy the Whitney
conditions. By Theorem \ref{th:1}, no $V_t$ is metrically conical.
\begin{theorem}
  The bi-Lipschitz type is non-constant in the Brian\c{c}on-Speder
  family: $V_t$ has a separating set for $t\ne0$ and none for $t= 0$.
\end{theorem}
 As
already mentioned, details will be in \cite{BFN-inprep}.


\begin{thebibliography}{BDM}
\bibitem{B99} Lev Birbrair Local bi-Lipschitz classification of
  $2$-dimensional semialgebraic sets.  Houston J. Math.  {\bf25}
  (1999), 453--472.
\bibitem{BB00}Lev Birbrair,  Jean-Paul Brasselet, Metric homology.  Comm. Pure
Appl. Math.  {\bf53}  (2000),  1434--1447.
\bibitem{BB02}
Lev Birbrair, Jean-Paul Brasselet, Metric homology for isolated
conical singularities.  Bull. Sci. Math.  {\bf126}  (2002), 87--95. 
\bibitem{BF00} Lev Birbrair, Alexandre C. G.  Fernandes, Metric theory
  of semialgebraic curves.  Rev. Mat. Complut.  13 (2000), 369--382.
\bibitem{BF08} Lev Birbrair and Alexandre Fernandes, Metric geometry of
  complex algebraic surfaces with isolated singularities, Comm. Pure
  Appl. Math. (to appear).
\bibitem{BFN08} Lev Birbrair, Alexandre Fernandes, and Walter D. Neumann,
  Bi-Lipschitz geometry of weighted homogeneous surface
  singularities, Math. Ann. (to appear).
\bibitem{BFN-inprep} Lev Birbrair, Alexandre Fernandes, and Walter D. Neumann,
  $\mu$--constant does not imply bi-Lipschitz triviality, in preparation
\bibitem{BGM91} JP Brasselet, M Goresky, R MacPherson,
Simplicial differential forms with poles. Amer. J. Math. {\bf113}
(1991),  1019--1052.
\bibitem{BS75} Jo\"el Brian\c{c}on and Jean-Paul Speder, La
  trivialit\'e topologique n'implique pas les conditions de Whitney.
  C. R. Acad. Sci. Paris S\'er. A-B {\bf280} (1975), 
  A365--A367.
\bibitem{federer} Herbert Federer, \emph{Geometric measure
    theory}. Die Grundlehren der mathematischen Wissenschaften, Band
  {\bf153} (Springer-Verlag New York Inc., New York 1969)
\bibitem{koike03} Satoshi Koike, The Briancon-Speder and Oka families
  are not bilipschits trivial, in \emph{Several topics in singularity
    theory}, Kyoto University research information repository {\bf
    1328} (2003), 165--173.\\
  \url{http://www.kurims.kyoto-u.ac.jp/~kyodo/kokyuroku/2003.html}
\bibitem{LR} J.-M. Lion, J.-P.  Rolin, Int\'egration des fonctions
  sous-analytiques et volumes des sous-ensembles
  sous-analytiques. Ann. Inst. Fourier (Grenoble) {\bf48} (1998),
   755--767.
\bibitem{valette}
Valette, Guillaume The link of the germ of a semi-algebraic metric space.  Proc. Amer. Math. Soc.  {\bf135}  (2007), 3083--3090. 
\end{thebibliography}
\end{document}